\documentclass[11pt]{amsart}
\usepackage{amsmath,amsfonts, mathtools,amssymb,latexsym,amssymb,amsthm,adjustbox,mathrsfs,amsbsy, bm,tikz-cd,indentfirst, makecell,graphicx, verbatim, calc, enumerate,xy}
\usepackage[top=1in,bottom=1.3in,left=1in,right=1in]{geometry}
\usepackage[colorlinks=true,linkcolor=blue]{hyperref}
\usepackage{listings}
\usepackage{xcolor}
\theoremstyle{plain}
\newtheorem{theorem}{Theorem}[section]
\newtheorem{lemma}[theorem]{Lemma}

\newtheorem{corollary}[theorem]{Corollary}
\newtheorem{proposition}[theorem]{Proposition}
\newtheorem{remark}[theorem]{Remark}

\newtheorem{question}[theorem]{Question}

\newtheorem{example}[theorem]{Example}
\newtheorem{point}[theorem]{}

\newcommand{\m}{\mathfrak{m}}
\newcommand{\n}{\mathfrak{n}}

\newcommand{\N}{\mathbb{N}}

\newcommand{\wt}{\widetilde }

\newcommand{\depth}{\operatorname{depth}}

\newcommand{\reg}{\operatorname{reg}}
\newcommand{\cx}{\operatorname{cx}}

\usepackage{graphicx}
\usepackage{subfigure}
\usepackage{epsfig}
\usepackage{epstopdf}
\usepackage{float}
\theoremstyle{plain}

\usepackage{xcolor}
\usepackage{titlesec}

\titleformat{name=\section}{}{\thetitle.}{0.5em}{\centering\scshape\large}

\begin{document}
\title{\large \textbf{Bounds on Hilbert coefficients of Cohen-Macaulay modules having finite projective dimension}}

\author{Samarendra Sahoo}
\email{samarendra.s.math@gmail.com}
\address{ Indian Institute of Technology Dharwad, Dharwad 580011, Karnataka, India}
\date{\today}


\begin{abstract}
Let $(A,\mathfrak{m})$ be a Gorenstein local ring with $G(A)$ Cohen-Macaulay, and let $M$ be a Cohen-Macaulay $A$-module of finite projective dimension. In \cite{Quasipure}, the authors proved that $e_1(M)\geq \binom{c+1}{2}$, where $c=\operatorname{reg}G(A)$ and $e_i(M)$ is the $i$th Hilbert coefficient of $M$. We first show that this bound remains valid when $A$ is Cohen-Macaulay. We then study upper bounds for $e_2(M)$ when $e_1(M)=\binom{c+1}{2}+i$ for $i=1,2$, and investigate the consequences of equality. In particular, we obtain depth properties and explicit descriptions of the $h$-polynomial of $G(M)$. Finally, we extend these results to strict complete intersection rings without assuming that $M$ has finite projective dimension.
\end{abstract}

\maketitle

\section{Introduction}

Bounding Hilbert coefficients is a classical problem in commutative algebra. Sharp bounds for these numerical invariants often provide valuable information about the underlying algebraic and homological structure of rings and modules. In particular, equality cases in such bounds can lead to strong structural consequences. Motivated by these phenomena, in this paper we study bounds for the second Hilbert coefficient of a Cohen-Macaulay module and investigate the structural consequences when these bounds are attained.

Let $(A,\mathfrak{m})$ be a local ring of dimension $d$ with residue field $k=A/\mathfrak{m}$, and let $M$ be a finitely generated $A$-module. Let $G(A)=\bigoplus_{n\geq 0}\mathfrak{m}^n/\mathfrak{m}^{n+1}$ be the associated graded ring of $A$, and let $G(M)=\bigoplus_{n\geq 0}\mathfrak{m}^nM/\mathfrak{m}^{n+1}M$ be the associated graded module of $M$, considered as a graded $G(A)$-module. We denote the length of an $A$-module $E$ by $\lambda_A(E)$.

Let $M$ be an $A$-module of dimension $r$. The Hilbert series of $M$ is given by
$$H_M(z)=\sum_{i\geq 0}\lambda_A(\mathfrak{m}^iM/\mathfrak{m}^{i+1}M)z^i
=\frac{h_M(z)}{(1-z)^r},$$ where $h_M(z)\in\mathbb{Z}[z]$ is called the $h$-polynomial of $M$. The integer
$e_i(M)={h_M^{(i)}(1)}/{i!}$ is called the $i$-th Hilbert coefficient of $M$, where $h_M^{(i)}(z)$ denotes the $i$-th derivative of $h_M(z)$.

Let $(A,\mathfrak{m})$ be a Gorenstein local ring such that $G(A)$ is Cohen-Macaulay, and let $M$ be a Cohen-Macaulay $A$-module of finite projective dimension. In Section 4 of \cite{Quasipure}, the authors establish the lower bound $e_1(M)\geq \binom{c+1}{2}$, where $c=\reg G(A)$ (for definition of regularity, see \ref{reg1}). Moreover, they show that equality in this bound imposes a strong structural condition: if $e_1(M)=\binom{c+1}{2}$, then the associated graded module $G(M)$ is Cohen-Macaulay.

This result highlights the connection between the first Hilbert coefficient and the structure of the associated graded module. In particular, the bound on $e_1(M)$ is expressed in terms of $\reg G(A)$, while equality forces $G(M)$ to be Cohen-Macaulay. This naturally leads to the question of whether such a phenomenon persists under weaker assumptions and whether bounds can be obtained for higher Hilbert coefficients. More precisely, we are interested in the following questions.

\begin{question}
\begin{enumerate}
    \item Can the Gorenstein hypothesis on $A$ in the above result be weakened to the assumption that $A$ is Cohen-Macaulay?
    \item Can one obtain bounds for the second Hilbert coefficient $e_2(M)$ in terms of $\reg G(A)$?
\end{enumerate}
\end{question}

In this paper, we answer both questions affirmatively. We establish upper bounds for the second Hilbert coefficient of a Cohen-Macaulay module under the hypotheses $e_1(M)=\binom{c+1}{2}+i$, where $i=1,2$, and investigate the structural consequences of the extremal cases. In particular, we show that when $e_2(M)$ attains the corresponding upper bound, the associated graded module $G(M)$ satisfies strong depth properties. We first establish the following lower bounds for $e_0(M)$ and $e_1(M)$, together with the Cohen-Macaulayness of $G(M)$ in the extremal case.

\begin{proposition}[Lemma \ref{lowerbound}, Proposition \ref{reg}] \label{r1}
    Let $(A,\mathfrak{m})$ be a Cohen-Macaulay local ring of dimension $d$ such that $G(A)$ is Cohen-Macaulay. Let $M$ be a Cohen-Macaulay $A$-module of dimension $r<d$ and finite projective dimension. Set $c=\operatorname{reg} G(A)$. Then $e_0(M)\geq\mu(M)+c$ and $e_1(M)\geq\binom{c+1}{2}$. If $e_1(M)$ attains its lower bound then $G(M)$ is Cohen-Macaulay.
\end{proposition}

The preceding proposition provides the starting point for our analysis of the cases in which $e_1(M)$ is close to its lower bound. To study these near-extremal cases, we reduce the problem to a suitable lower-dimensional setting using an $A\oplus M$-superficial sequence. The deviation from the extremal case is measured by the invariant $\gamma$, which is bounded above by the excess of $e_1(M)$ over its lower bound. More precisely, with the same hypothesis as in Proposition \ref{r1}, let $x_1,\ldots,x_{r-2},x,y$ be an $A\oplus M$-superficial sequence. Set $c=\reg G(A)$, $L=M/(x_1,\ldots,x_{r-2},x)M$, $N=L/yL$ and $(B,\n)=(A/(x_1,\ldots,x_{r-2},x),\m/(x_1,\ldots,x_{r-2},x))$. If $e_1(M)=\binom{c+1}{2}+s$, we prove that
$\displaystyle \gamma=\sum_{i\geq 1} \lambda_B \left(\dfrac{\n^{i+1}L:y}{\n^iL }\right)\leq s$ (see Lemma \ref{kam}).

We first consider the case where $e_1(M)$ exceeds its lower bound by one. The bound on $\gamma$ then leaves only two possibilities, namely $\gamma=0$ and $\gamma=1$. These two cases lead to distinct structural behaviours for $G(M)$ and allow us to obtain corresponding upper bounds for $e_2(M)$. We prove the following result.

\begin{theorem}[Theorem \ref{main+1}] \label{r2}
     (with the same hypothesis as in Proposition \ref{r1}) If $e_1(M)=\binom{c+1}{2}+1$ then the following hold. 
     \begin{enumerate}
         \item If $\gamma=0$ then $G(M)$ is Cohen-Macaulay. Moreover, $e_2(M)=\binom{c+1}{3}$ with the $h$-polynomial $h_M(z) = \mu(M)+2z+z^2+\ldots+ z^{c}.$
         \item  If $\gamma=1$ then $e_2(M)\leq \binom{c+1}{3}+(c-1).$ Further, if equality holds, then $\depth G(M) = r-1$ and $h_M(z)=\mu(M)+z+\ldots+z^{c-2}+2z^{c}$.
     \end{enumerate} 
\end{theorem}

We next consider the case where $e_1(M)$ exceeds its lower bound by two. In this case, the inequality $\gamma\leq 2$ gives three possible values for $\gamma$. We analyze each case separately and obtain upper bounds for $e_2(M)$. As in the preceding case, equality in these bounds yields precise information about the depth and the $h$-polynomial of $G(M)$. We prove the following theorem.

\begin{theorem} [Theorem \ref{main}]\label{r3}
    (with the same hypothesis as in Proposition \ref{r1})  If $e_1(M)=\binom{c+1}{2}+2$ then the following hold. 
     \begin{enumerate}
         \item  If $\gamma=0$ then $G(M)$ is Cohen-Macaulay. Furthermore,   $e_2(M)=\binom{c+1}{3}$ with the $h$-polynomial $h_M(z)=\mu(M)+3z+z^2+\ldots +z^c$. 
         \item   If $\gamma=1$ then $e_2(M)\leq \binom{c+1}{3}+(c-1).$ Further, if equality holds, then $\depth G(M)= r-1$ and $h_M(z)=\mu(M)+2z+z^2+\cdots+z^{c-2}+2z^c$.
          \item  If $\gamma=2$ then $e_2(M)\leq \binom{c+1}{3}+2(c-1).$ Further, if equality holds, then $\depth G(M)= r-1$ and $h_M(z)= \mu(M)+z+z^2+\cdots+z^{c-2}+2z^c$.
     \end{enumerate} 
\end{theorem}

Next, we extend these results to strict complete intersection rings without assuming that $M$ has finite projective dimension. In this setting, we use the complexity of $M$, which measures the asymptotic growth of the Betti numbers of $M$ (see \ref{complexity} for the definition). We establish analogous bounds and structural consequences in this setting and prove the following theorem.

\begin{theorem}[Theorem \ref{ci1}, \ref{ci2}]\label{r4}
     Let  $(Q,\n)$ be a regular local ring with $f_1,\ldots ,f_l\in \n^2$ of order $s$ such that $f_1^*,\ldots ,f_l^* $ is a $G(Q)$-regular sequence.  Let $A=Q/(f_1,\ldots ,f_l)$ and $M$ be Cohen Macaulay $A$-module of dimension $r$ with $\operatorname{cx}_A(M)=t< l$. Let $e_1(M)= \binom{c+1}{2}+i$, where $i=1,2$. Then the following hold.
\begin{enumerate}
    \item If $G(M)$ is Cohen-Macaulay, then $e_2(M)=\binom{c+1}{3}$.
    \item If $i=1$ and $G(M)$ is not Cohen-Macaulay, then $e_2(M)\leq \binom{c+1}{3}+(c-1).$ Further, if equality holds, then $\depth G(M)=r-1.$
    \item If $i=2$ and $G(M)$ is not Cohen-Macaulay, then $e_2(M)\leq \binom{c+1}{3}+2(c-1).$
\end{enumerate}
\end{theorem}

\textbf{Organization of the paper.} This paper is organized into four sections. Section 2 presents the necessary preliminary results. In Section 3, we first prove Proposition \ref{r1}, then establish Lemma \ref{kam}, and subsequently prove Theorems \ref{r2} and \ref{r3}, along with examples supporting the results in the case $\gamma=0$. In Section 4, we prove Theorem \ref{r4} along with examples in the Cohen-Macaulay case.

\section{Preliminaries}
Throughout this paper, all rings considered are Noetherian local rings, and all modules considered (unless stated otherwise) are finitely generated. All the examples discussed in this paper are verified by Macaulay2 \cite{macaulay2}. In this section, we discuss some preliminaries that we use throughout the paper.

\begin{point} Superficial elements \\
     \normalfont
   An element $x\in  \m$ is called $M$-superficial with respect to $\m$ if there exists $c\in \N$ such that for all $n \geq c$, $(\m^{n+1}M :_Mx)\cap \m^cM = \m^nM$. If depth $M>0$ then one can show that a $M$-superficial element is $M$-regular. Furthermore $(\m^{n+1}M :_Mx) = \m^nM$ for $n\gg0.$  If the residue field is infinite then linear superficial elements exist.

    A sequence $x_1, \ldots ,x_r$  in $(A, \m)$ is said to be $M$-superficial sequence if $x_1$ is $M$-superficial and $x_i$ is $M/(x_1,\ldots ,x_{i-1})M$-superficial for $2\leq i \leq r .$ For more details, see \cite[Proposition 8.5.7]{HS}.
\end{point}

\begin{point} Sally's Descent \cite[Theorem 8]{HCCMM}\label{sally} 
\normalfont
$\,$ Let $(A,\m)$ be local ring, $M$ an $A$-module of dimension $d$. Let $x_1,\ldots ,x_r$ is an $M$-superficial sequence with $r< d$. Set $B=A/(x_1,\ldots ,x_r)$ and $N=M/(x_1,\ldots ,x_r)M$.
 Then $$\operatorname{depth }_{G(B)}G(N)\geq 1 \operatorname{ iff } \operatorname{depth }_{G(A)}G(M)\geq r+1.$$
\end{point}

\begin{point}\label{basechange}
Base change \\
\normalfont
   Let $\phi:(A,\mathfrak{m})\to (A',\mathfrak{m}')$ be a local ring homomorphism. Assume $A'$ is a faithfully flat $A$-algebra with $\mathfrak{m}A'=\mathfrak{m}'$. Let $M$ be an $A$-module. Set $M'=M\otimes A'$. In these cases, it can be shown that
\begin{enumerate}
    \item $\lambda_A(M)=\lambda_{A'}(M')$.
    \item $H_M(n)(=\lambda_{A}(\m^nM/\m^{n+1}M))=H_{M'}(n)(=\lambda_{A'}({\m'}^nM'/{\m'}^{n+1}M'))$ for all $n\geq 0$ and hence $H_M(z)=H_{M'}(z).$
    \item dim$_{A}M=$ dim$_{A'}M'$ and $\text{ depth}_A M= \text{ depth}_{A'} M'$.
    \item depth$_{G(A)}G(M)=\text{ depth}_{G(A')} G(M')$.
    \item $A$ is a local complete intersection iff $A'$ is a local complete intersection.
    \item Let $\beta^A_i(M)$ be the $i$th betti number of $M.$ $\beta^A_i(M)=\beta^{A'}_i(M')$. This implies $\operatorname{cx}_A(M)=\operatorname{cx}_{A'}(M')$, where $\operatorname{cx}_A(M)$ is the complexity of $M$ (for definition of complexity see \ref{complexity}).
\end{enumerate}
\end{point}
The specific base changes we do are the following:
\begin{enumerate}
    \item[(i)] $A'=\hat{A}$ the completion of $A$ with respect to the maximal ideal.
    \item[(ii)] $A'=A[[X]]_S$, where $S=A[[X]]\setminus \m A[[X]].$ The maximal ideal of $A'$ is $\m '=\m A'$ and the residue field of $A'$ is $K=k((X)),$ which is uncountable.
\end{enumerate}
In most proofs, we need the residue field to be infinite for the existence of superficial elements, which can be done by the above extension. Therefore, for simplicity,  throughout this paper we assume that the residue field $k$ is infinite.

\begin{point} Castelnuovo-Mumford regularity \label{reg1} \\
    \normalfont
Let $R=\bigoplus_{n\geq 0}R_n$ be a standard $\mathbb{N}$-graded Noetherian ring with $R_0=A$, and let $\mathfrak{M}$ denote its graded maximal ideal. For a finitely generated graded $R$-module $E$, let $H^i_{\mathfrak{M}}(E)$ denote the $i$-th graded local cohomology module of $E$. We set $\operatorname{end}(X)=\max\{n\mid X_n\neq0\},$
with the convention that $\operatorname{end}(0)=-\infty$. Define
$a_i(E)=\operatorname{end}
\left(H^i_{\mathfrak{M}}(E)\right).$
The Castelnuovo-Mumford regularity of $E$ is defined as:
\[
\operatorname{reg} E
=
\max\{a_i(E)+i\mid 0\leq i\leq\dim E\}.
\]
\end{point}

\begin{point}\label{rit}
\normalfont

Let $I$ be an ideal of a ring $A$, and let $M$ be an $A$-module. We consider the ideal $\wt{I^kM} = \bigcup_{n\geq 0} (I^{k+n}M :_M I^n) = (I^{k+n}M :_M I^n)$ for $n \gg 0$. One can readily verify that the companion family of submodules $\{\widetilde{I^nM}\}_{n\geq 0}$ likewise forms an $I$-filtration.
\end{point}

\begin{remark}
\label{rmk:tildeI^nM}
\normalfont

\begin{enumerate}
    \item If $\operatorname{grade}(I, M) > 0$, then $I^nM = \widetilde{I^nM}$ for all $n \gg 0$. Consequently, the family $\{\widetilde{I^nM}\}_{n\geq 0}$ constitutes an $I$-stable filtration (cf.\ \cite[2.2]{tildeI^nM}).
    \item If $\operatorname{grade}(G_I(A)_+, G_I(M)) > 0$, then the equality $I^nM = \widetilde{I^nM}$ holds for all $n \geq 1$ (see \cite[Fact 9]{Heinzer}).
\end{enumerate}
Now define $s^*(I,M)=\text{min}\{ n\, |\, \wt{I^mM}=I^mM \text{ for all }m\geq n\}.$ This is also called as the stability index of $M$ w.r.t. $I.$
\end{remark}

We end this section by recalling the definition of the reduction number and a result which compares the reduction number and the stability index in dimension one.

 Let $(A,\m)$ be a local ring and $I$ an ideal of $A.$ A reduction of $I$ is an ideal $J \subseteq I$ such that $JI^n =I^{n+1}$ for some $n \in N$. If $J$ is a reduction of $I$, the reduction number of $I$ with respect to $J$ is defined as: $$\text{r}_J(I) =\text{min}\{n\, |\, JI^n = I^{n+1}\}.$$ A reduction is minimal if it is minimal with respect to inclusion. The reduction number of $I$ is defined as $\text{r}(I) = \text{min}\{\text{r}_J(I) \, |\, J \text{ is a minimal reduction of }  I\}$.

 \begin{proposition}\cite[Proposition 4.3]{q-hilbert} \label{d1}
      Let $A$ be a local ring, $M$ be an $A$-module, and $I$ an ideal of $A$ having principal reduction such that grade$(I,M)>0$. Then $s^*(I,M)\leq r(I)$.
 \end{proposition}

\section{Hilbert coefficient of Cohen-Macaulay modules}
 In [\cite{nawaz}, Theorem B], the authors proved that if $(A,\m)$ is a Cohen-Macaulay local ring and $G(A)$ is Cohen Macaulay,  then for all finite length $A$-modules $E$ of finite projective dimension $\ell\ell(E)\geq \operatorname{reg}G(A)+1$ (here $\ell\ell(E)$ denotes the Loewy length of $E$.  Set $c=\operatorname{reg}G(A)$. Using this result, we first generalize the results of (\cite{Quasipure}, Section 4) and then give an upper bound to $e_2(M)$ in terms of Castelnuovo-Mumford regularity.
\begin{point}
\normalfont
Recall that the \emph{Loewy length} of a finite-length $A$-module $E$ is defined by
$$\ell\ell(E)=\min\{i\mid \mathfrak{m}^iE=0\}$$.
Let $G(A)_+$ denote the irrelevant maximal ideal of $G(A)$. For each $i$, let
$H^i(G(A))$ denote the $i$-th local cohomology module of $G(A)$ with respect to
$G(A)_+$. The \emph{Castelnuovo-Mumford regularity} of $G(A)$ is defined as
$$\operatorname{reg}G(A)=\max\{i+j\mid H^i(G(A))_j\neq 0\}$$.

Let $x\in A$. For $x\in \m^r\setminus \m^{r+1}$, set $x^*=x+\m^{r+1}$ as an homogenous element of $G(A)$. We call $x^*$ as the initial form of $x.$
\begin{remark}\label{regmod}
    Let $(A,\m)$ be a Cohen-Macaulay local ring of dimension $d\geq 1$ with $G(A)$ Cohen-Macaulay. Let $x^*$ be a linear $G(A)$-regular element. It can be easily checked that $\operatorname{reg}G(A)=\operatorname{reg}G(A/xA).$ 
\end{remark}
\end{point}

The following result generalizes Lemma 4.3 of \cite{Quasipure}. Its proof follows a similar argument to that of Lemma 4.3 in \cite{Quasipure}, and we include the proof for completeness, as the argument will be used repeatedly in the sequel.

\begin{lemma}
\label{lowerbound}
Let $(A,\mathfrak{m})$ be a Cohen-Macaulay local ring of dimension $d$ such that $G(A)$ is Cohen-Macaulay. Let $M$ be a Cohen-Macaulay $A$-module of dimension
$r<d$ and finite projective dimension. Set
$c=\operatorname{reg} G(A)$. Then
$e_0(M)\geq\mu(M)+c$ and $e_1(M)\geq\binom{c+1}{2}$.
\end{lemma}

\begin{proof}
Let $x_1,\ldots,x_r\in\mathfrak{m}\setminus\mathfrak{m}^2$ be an
$A\oplus M$-superficial sequence. Set
$B=A/(x_1,\ldots,x_r)$ and $N=M/(x_1,\ldots,x_r)M$.
Then $N$ has finite length. Let $\ell\ell(N)=t$, so that
$ G(N)=\bigoplus_{i=0}^{t-1}\mathfrak{m}^iN/\mathfrak{m}^{i+1}N$.
Set $ a_i=\lambda(\mathfrak{m}^iN/\mathfrak{m}^{i+1}N)$. The Hilbert series of $N$ is
$$H_N(z)=\mu(N)+a_1z+\cdots+a_{t-1}z^{t-1}.$$
Hence

$$e_0(N)=\mu(N)+a_1+\cdots+a_{t-1}
\geq\mu(N)+t-1
\geq\mu(N)+c.$$
The last equality is by Theorem B of \cite{nawaz}. Moreover,
\[
\begin{aligned}
e_1(N)
&=a_1+2a_2+\cdots+(t-1)a_{t-1}\\
&\geq 1+2+\cdots+(t-1)\\
&\geq 1+2+\cdots+c
=\binom{c+1}{2}.
\end{aligned}
\]
The second-to-last inequality is by Theorem B of \cite{nawaz}. By \cite[Corollary 10]{HCCMM} and Remark~\ref{regmod}, we have
$\mu(M)=\mu(N)$, $e_0(M)=e_0(N)$, $e_1(M)\geq e_1(N)$, and $\operatorname{reg} G(A)=\operatorname{reg} G(B)$. Therefore,
$e_0(M)\geq\mu(M)+\operatorname{reg} G(A)$ and
$e_1(M)\geq\binom{c+1}{2}$, as desired.
\end{proof}

The next result is a generalization of Theorem 4.4 of \cite{Quasipure} and discuss about the consequences when $e_1(M)=\binom{c+1}{2}.$
\begin{proposition}
\label{reg}
    (with the same hypothesis as in Lemma \ref{lowerbound}) If $e_1(M)=\binom{c+1}{2}$, then $G(M)$ is Cohen-Macaulay.
\end{proposition}
\begin{proof}
 Let $x_1,\ldots,x_r\in\mathfrak{m}\setminus\mathfrak{m}^2$ be an $A\oplus M$-superficial sequence, and set
$L=M/(x_1,\ldots,x_{r-1})M$. By \cite[Corollary 10]{HCCMM}, we have $e_1(M)=e_1(L)$ and $e_1(L)\geq e_1(L/x_rL)$. Thus, by the hypothesis,
$$\binom{c+1}{2}=e_1(L)\geq e_1(L/x_rL)\geq\binom{c+1}{2}.$$ It follows that $e_1(L)=e_1(L/x_rL)$. Again, by \cite[Corollary 10]{HCCMM}, $G(L)$ is Cohen-Macaulay. Hence, by Sally's descent (see \ref{sally}), $G(M)$ is Cohen-Macaulay.
\end{proof}

We now give a simple example illustrating the equality case in Theorem~\ref{reg}.

\begin{example}\label{ex1}
\normalfont
Let $A=k[[x,y,z]]/(x^t)$ and $M=A/(y)=k[[x,y,z]]/(x^t,y)$, where $k$ is an infinite field. Then $A$ is Cohen-Macaulay of dimension $2$, while $M$ is Cohen-Macaulay of dimension $1$. Since $y$ is a non-zero-divisor on $A$, $M$ has finite projective dimension over $A$.

Moreover, $G(A)\cong k[X,Y,Z]/(X^t)$ and hence $\operatorname{reg} G(A)=t-1$. On the other hand,
$G(M)\cong k[X,Z]/(X^t)$, and its Hilbert series is
$\displaystyle H_{M}(z)=\frac{1+z+\ldots +z^{t-1}}{(1-z)}$. Thus $h_M(z)=1+z+\ldots +z^{t-1}$ and
$e_1(M)=\binom{t}{2}=\binom{c+1}{2}$, where $c=\operatorname{reg} G(A)=t-1$. Therefore, by Theorem~\ref{reg}, $G(M)$ is Cohen-Macaulay.
\end{example}

\begin{lemma}\label{kam}
   Let $(A,\mathfrak{m})$ be a Cohen-Macaulay local ring and $M$ be a Cohen-Macaulay $A$-module of dimension $r$. Let $x_1,\ldots ,x_{r-2},x,y$ be $A\oplus M$-superficial sequence. Set  $L=M/(x_1,\ldots ,x_{r-2},x)M$, $N=L/yL$, and  $(B,\n)=(A/(x_1,\ldots ,x_{r-2},x),\m/(x_1,\ldots ,x_{r-2},x))$. Then
   
   $$e_2(M)\leq e_2(N) + \sum_{i\geq 1}i\lambda_B\left(\frac{{\n}^{i+1}L:y}{{\n}^iL}\right) $$ Moreover, if equality holds, then depth $G(M)\geq r-1$.
\end{lemma}
\begin{proof}
  Set $ s(z)=\sum_{i\geq 0}\lambda_B \left(\dfrac{\n^{i+1}L:y}{\n^iL}\right)z^i$. Note that $\dim(L)=1.$ By (\cite{HCCMM}, Corollary 10(3)), we have 
    \begin{equation}\label{eq12}
        \begin{split}
           & h_{L}(z)=h_{N}(z)+(z-1)s(z)\\
           \implies &  h^{(1)}_L(z)=h^{(1)}_{N}(z)+(z-1)s^{(1)}(z)+s(z)\\
           \implies & h^{(2)}_{L}(z)=h^{(2)}_{N}(z)+(z-1)s^{(2)}(z)+2s^{(1)}(z)\\
           \implies & e_2(L)=e_2(N)+s^{(1)}(1)=e_2(N)+\sum_{i\geq 1}i\lambda_B \left(\dfrac{\n^{i+1}L:y}{\n^iL}\right).
        \end{split}
    \end{equation}
Set  $M_1=M/(x_1,\ldots ,x_{r-2})M$, $(A_1,\m_1)=(A/(x_1,\ldots ,x_{r-2}),\m/(x_1,\ldots ,x_{r-2}))$, and  $$ r(z)=\sum_{i\geq 0}\lambda_{A_1} \left(\dfrac{\m^{i+1}_1M_1:x}{\m^i_1M_1}\right)z^i.$$ Note that $\dim(M_1)=2.$ By (\cite{HCCMM}, Corollary 10(5,6)), we have 

\begin{equation}
    \begin{split}
        e_2(M)=e_2(M_1) &=e_2(L)-r(1)\\ & = e_2(N)+\sum_{i\geq 1}i\lambda_B \left(\dfrac{\n^{i+1}L:y}{\n^iL}\right)-r(1) \\ & \leq e_2(N)+\sum_{i\geq 1}i\lambda_B \left(\dfrac{\n^{i+1}L:y}{\n^iL}\right).
    \end{split}
\end{equation}
Furthermore, if equality holds, then $r(1)=0$. Consequently, $(\mathfrak{m}^{i+1}_1M_1:x)=\mathfrak{m}^i_1M_1$ for all $i\geq1$. Therefore $\operatorname{depth}G(M_1)\geq1$ and hence by sally's descent (see \ref{sally}), $\depth G(M)\geq r-1.$ 

\end{proof}

The next result examines the consequences of the next-to-extremal cases, where $e_1(M)=\binom{c+1}{2}+s$ and $s\geq 1$.

\begin{lemma}\label{implemma}
     Let $(A,\m)$ be a Cohen-Macaulay local ring of dimension $d$ with $G(A)$ Cohen-Macaulay. Let $M$ be a Cohen-Macaulay $A$-module of dimension $r<d$ with finite projective dimension. Let $x_1,\ldots ,x_{r-2},x,y$ be $A\oplus M$-superficial sequence. Set $c=\reg G(A)$, $L=M/(x_1,\ldots ,x_{r-2},x)M$, $N=L/yL$ and  $(B,\n)=(A/(x_1,\ldots ,x_{r-2},x),\m/(x_1,\ldots ,x_{r-2},x))$. If $e_1(M)=\binom{c+1}{2}+s$ then $$ \gamma=\sum_{i\geq 1} \lambda_B \left(\dfrac{\n^{i+1}L:y}{\n^iL }\right)\leq s.$$
\end{lemma}
\begin{proof}
     By (\cite{HCCMM}, Corollary 10), we have $e_1(M)=e_1(L)$ and $ e_1(L)=e_1(N)+\sum_{i\geq 1}\lambda_B \left(\dfrac{\n^{i+1}L:y}{\n^iL}\right)$. Set $c=\reg G(A).$ Now, from the proof of Lemma \ref{lowerbound}, we obtain
    \begin{equation*}
        \begin{split}
            \binom{c+1}{2}+s=e_1(M)=e_1(L)&=  e_1(N)+\sum_{i\geq 1}\lambda_B \left(\dfrac{\n^{i+1}L:y}{\n^i}L\right) \\ & \geq \binom{c+1}{2}+\sum_{i\geq 1}\lambda_B \left(\dfrac{\n^{i+1}L:y}{\n^iL}\right).
        \end{split}
    \end{equation*}
This proves the result.% This implies that $$\gamma = \sum_{i\geq 1}\lambda_B \left(\dfrac{\n^{i+1}L:y}{\n^iL}\right)\leq s.$$
\end{proof}

We now give an example illustrating the hypothesis of Lemma~\ref{implemma}.
\begin{example}
    \normalfont
Let $A=k[[x,y,z]]/(x^2)$ and  $M_1=A/(y)=k[[x,y,z]]/(x^2,y)$, where $k$ is an infinite field. Set $M=M_1^{s+1}$. By Example~\ref{ex1}, $M_1$ has finite projective dimension over $A$, and hence so does $M$. Moreover, $M$ is a Cohen-Macaulay $A$-module of dimension $1$.

We have $G(A)\cong k[X,Y,Z]/(X^2)$ with $\operatorname{reg} G(A)=1$ and $G(M_1)\cong k[X,Z]/(X^2)$. Since $G(M)=G(M_1)^{s+1}$ and the Hilbert series is additive with respect to direct sums, we have
$\displaystyle H_M(z)=(s+1)H_{M_1}(z)=\frac{(s+1){(1+z)}}{(1-z)}$. Hence, $h_M(z)=(s+1)(1+z)$, and therefore $e_1(M)=s+1=\binom{2}{2}+s=\binom{c+1}{2}+s$.
\end{example}

We now investigate the structure arising in the case $e_1(M)=\binom{c+1}{2}+1$.

\begin{theorem}\label{main+1}
    (with the same hypothesis as in Lemma \ref{implemma}) If $e_1(M)=\binom{c+1}{2}+1$ then the following hold. 
     \begin{enumerate}
         \item \label{1+11} If $\gamma=0$ then $G(M)$ is Cohen-Macaulay. Moreover, $e_2(M)=\binom{c+1}{3}$ with the $h$-polynomial $h_M(z) = \mu(M)+2z+z^2+\ldots+ z^{c}.$
         \item \label{1+12} If $\gamma=1$ then $e_2(M)\leq \binom{c+1}{3}+(c-1).$ Further, if equality holds, then $\depth G(M)= r-1$ and $h_M(z)=\mu(M)+z+\ldots+z^{c-2}+2z^{c}$.
     \end{enumerate} 
\end{theorem}

\begin{proof}

  (\ref{1+11}) If $\gamma=0$, then $\mathfrak{n}^{s+1}L:y=\mathfrak{n}^sL$ for all $s\geq1$. It follows that $\operatorname{depth}G(L)\geq1$, and hence, by Sally's descent (see \ref{sally}), $G(M)$ is Cohen-Macaulay. Thus, $h_M(z)=h_N(z)$ (see \cite[Theorem 8, Corollary 11]{HCCMM}). Consequently, $e_1(N)=\binom{c+1}{2}+1$. Therefore, by the Lemma \ref{lowerbound}, the only possible case is $a_1=2$ and $a_i=1$ for all $1<i\leq c=t-1$. Hence, $h_M(z)=h_N(z)=\mu(M)+2z+z^2+\cdots+z^c$. Therefore, $$e_2(M)=\frac{h_N^{(2)}(1)}{2}=\frac{1}{2}\big(1\cdot2+2\cdot3+\cdots+(c-1)c\big)=\binom{c+1}{3}$$.

\noindent
  (\ref{1+12}) If $\gamma= 1$ then $G(L)$ is not Cohen-Macaulay. Therefore, $G(M)$ is not Cohen-Macaulay and hence $e_1(M)\neq e_1(N)$ (see \cite[Corollary 10]{HCCMM}). It follows from the proof of Lemma \ref{lowerbound} that $e_1(N)=\binom{c+1}{2}.$ Thus, the only possible case is $a_i=1$ for all $1\leq i\leq c=t-1.$ Therefore, $h_N(z)=\mu(M)+z+\ldots+ z^{c}$ and hence $e_2(N)=\binom{c+1}{3}.$

Suppose for $i=\alpha$, $\lambda_B \left(\dfrac{\n^{i+1}L:y}{\n^iL}\right)=1.$ By (\cite{Zulfeqarr}, 8.3), we have $s^*(\n,L)=\min\{s\,|\,(\n^{m+1}L:y)=\n^mL \text{ for all }m\geq s\}.$ Thus, $\alpha=s^*(\n,L)-1.$
From Lemma \ref{kam}, we obtain

\begin{align*}
      e_2(M) & \leq e_2(N) + \sum_{i\geq 1}i\lambda_B\left(\frac{{\n}^{i+1}L:y}{{\n}^iL}\right)\\ & = e_2(N) + \alpha \\ & =e_2(N)+s^*(\n,L)-1 \\ & \leq e_2(N)+r(\n)-1 & \text{ (By Proposition \ref{d1})} \\ & = e_2(N)+ \text{reg }G(B)-1 \\& = e_2(N)+ c-1 & \text{ (By Remark \ref{regmod})}\\ &= \binom{c+1}{3}+c-1,
\end{align*}

where the equality $r(\mathfrak{n})=\operatorname{reg} G(B)$ follows from the Cohen-Macaulayness of $G(A)$. By Lemma~\ref{kam}, equality implies that $\operatorname{depth}G(M)\geq r-1$. Since $\alpha\neq0$, we obtain $\operatorname{depth}G(M)=r-1$. Moreover, if equality holds, then $\alpha=c-1$. It follows that
$h_M(z)=h_L(z)=h_N(z)-(1-z)z^{c-1}
=\mu(M)+z+\cdots+z^{c-2}+2z^c$.
This proves the result.
\end{proof}

We now give an example in which equality holds in Theorem~\ref{main+1}(1).

\begin{example}
\normalfont
Let $A=k[[x,y,z,w,u]]/(x^2,y^2)$, where $k$ is an infinite field, and $M=A/(z)$. By Example~\ref{ex1}, $M$ has finite projective dimension over $A$. Moreover, $M$ is a Cohen-Macaulay $A$-module of dimension $2$. We have $G(A)\cong k[X,Y,Z,W,U]/(X^2,Y^2)$ with $c=\operatorname{reg} G(A)=2$, and
$G(M)\cong k[X,Y,W,U]/(X^2,Y^2)$. The Hilbert series of $M$ is
$\displaystyle H_M(z)=\frac{(1+z)^2}{1-z}$. Hence, $h_M(z)=1+2z+z^2$, and therefore
$e_1(M)=4=\binom{3}{2}+1=\binom{c+1}{2}+1$ and
$e_2(M)=1=\binom{c+1}{3}$.
\end{example}

We next turn to the case $e_1(M)=\binom{c+1}{2}+2$ and investigate its consequences. We need the following result to prove the next theorem.

\begin{lemma}\label{length}
Let $(A,\m,k)$ be a local ring and let $E$ be a finitely generated $A$-module. Then $$\lambda_A({\m E}/{\m^2 E})=1\implies \lambda_A({\m^2 E}/{\m^{3} E})\leq 1.$$
\end{lemma}

\begin{proof}
Choose $x\in \m$ and $u\in E$ such that
$xu\notin \m^2E$. Since $\lambda_A({\m E}/{\m^2 E})=1$, we have $\m E=A(xu)+\m^2E.$ Hence $\m^2E=\mathfrak m(xu)+\m^3E.$ For any $y\in\m$, write $yu=a.xu+v$ for some $a\in A$ and $v\in\m^2E$. This implies $y(xu)=a.x^2u+xv\equiv a.x^2u \pmod{\m^3E}$. Thus $\m^2E/\m^3E$ is generated, as a $k$-vector space, by the class of $x^2u$. Therefore $\lambda_A({\m^2E}/{\m^3E})\leq 1.$
\end{proof}

\begin{theorem}\label{main}
    (with the same hypothesis as in Lemma \ref{implemma}) If $e_1(M)=\binom{c+1}{2}+2$ then the following hold. 
     \begin{enumerate}
         \item \label{2+11} If $\gamma=0$ then $G(M)$ is Cohen-Macaulay. Furthermore,   $e_2(M)=\binom{c+1}{3}$ with the $h$-polynomial $h_M(z)=\mu(M)+3z+z^2+\ldots +z^c$. 
         \item \label{2+12}  If $\gamma=1$ then $e_2(M)\leq \binom{c+1}{3}+(c-1).$ Further, if equality holds, then $\depth G(M)= r-1$ and $h_M(z)=\mu(M)+2z+z^2+\cdots+z^{c-2}+2z^c$.
          \item \label{2+13} If $\gamma=2$ then $e_2(M)\leq \binom{c+1}{3}+2(c-1).$ Further, if equality holds, then $\depth G(M)= r-1$ and $h_M(z)= \mu(M)+z+z^2+\cdots+z^{c-2}+2z^c$.
     \end{enumerate} 
\end{theorem}
\begin{proof}
   (\ref{2+11}) If $\gamma=0$, then, by the proof of Theorem~\ref{main+1}, $G(M)$ is Cohen-Macaulay. It follows that $h_M(z)=h_N(z)$ and $ e_1(M)=e_1(N)=\binom{c+1}{2}+2$. Therefore, by the Lemma \ref{lowerbound}, there are exactly two possible $h$-polynomials for $M$, namely $h_M(z)=\mu(M)+3z+z^2+\cdots+z^c$ or $h_M(z)=\mu(M)+z+2z^2+z^3+\cdots+z^c$. Replacing $E$ by $N$ in Lemma \ref{length}, we obtain that the second case is not possible. Therefore, $h_M(z)=\mu(M)+3z+z^2+\cdots+z^c$ and $ e_2(M)=\binom{c+1}{3}$.

\noindent
   (\ref{2+12})  If $\gamma=1$, then $G(L)$ is not Cohen-Macaulay. Therefore, $G(M)$ is not Cohen-Macaulay. By Lemma \ref{lowerbound} and the proof of Theorem \ref{main+1}, we have $ e_1(N)=\binom{c+1}{2}+1$ and $h_N(z)=\mu(M)+2z+z^2+\ldots+ z^{c}.$ Hence $ e_2(N)=\binom{c+1}{3}$. By the same argument as in the proof of Theorem \ref{main+1}, we obtain $  e_2(M)\leq \binom{c+1}{3}+(c-1)$. Further, if equality holds, then $\depth G(M)= r-1$ and $h_M(z)=h_L(z)=h_N(z)-(1-z)z^{c-1}
=\mu(M)+2z+z^2+\cdots+z^{c-2}+2z^c$.

\noindent
(\ref{2+13}) If $\gamma =2$, then $G(L)$ is not Cohen-Macaulay. Therefore, $G(M)$ is not Cohen-Macaulay. By Lemma \ref{lowerbound} and the proof of Theorem \ref{main+1}, we have $ e_1(N)=\binom{c+1}{2}$ and $h_N(z)=\mu(M)+z+z^2+\ldots+ z^{c}.$ Hence $ e_2(N)=\binom{c+1}{3}$. 

Since $\gamma=2$, we have two possibilities,  for some $i=\alpha$, $\lambda_B \left(\dfrac{\n^{i+1}L:y}{\n^iL}\right)=2.$ or there is $\beta < \alpha$, such that  $\lambda_B \left(\dfrac{\n^{i+1}L:y}{\n^iL}\right)=1$ for $i=\beta ,\alpha$. In both the case,  we have $\displaystyle  \sum_{i\geq 1}i\lambda_B\left(\frac{{\n}^{i+1}L:y}{{\n}^iL}\right)\leq 2\alpha$. Note that, in the second case, the inequality is strict.

Now, By the same argument as in the proof of Theorem \ref{main+1}, we obtain $ e_2\leq \binom{c+1}{3}+2(c-1)$. Furthermore, if equality holds (i.e. equality in the first case), then by Lemma \ref{kam}, depth $G(M)=r-1$ and $\alpha=c-1.$ This implies that $h_M(z)=h_L(z)=h_N(z)-(1-z)z^{c-1}
=\mu(M)+z+z^2+\cdots+z^{c-2}+2z^c$.
\end{proof}

\begin{remark}
   Under the same hypotheses as in Lemma~\ref{implemma}, if $e_1(M)=\binom{c+1}{2}+s$ for some $s\geq3$, then we obtain a similar upper bound for $e_2(M)$ as in Theorem \ref{main+1}, \ref{main}. Since $e_2(M)$ is always non-negative (see \cite[Proposition 3.1]{rossi2010hilbert}), there are only finitely many possibilities for $e_2(M)$ whenever $e_1(M)=\binom{c+1}{2}+s$ for some $s\geq 0.$
\end{remark}

\begin{corollary}\label{reg2}
     Let $(A,\m)$ be a regular local ring of dimension $d$ and $M$ be a Cohen-Macaulay $A$-module of dimension $r<d$. If $e_1(M)=\binom{c+1}{2}+i$, where $i=1,2$, then $G(M)$ is Cohen-Macaulay and $e_2(M)=0$. 
\end{corollary}
\begin{proof}
Since $A$ is a regular local ring, we have $c=\reg G(A)=0$ and projective dimension of $M$ is finite. Hence, the case $\gamma>0$ in Theorems \ref{main+1} and \ref{main} cannot occur, as it would contradict \cite[Proposition 3.1]{rossi2010hilbert}. The result now follows directly from Theorems \ref{main+1} and \ref{main}.
\end{proof}

We now present an example in which equality holds in Theorem~\ref{main+1}(1). 

\begin{example}
\normalfont
Let $A=k[[x,y,z,w,u,v]]/(x^2,y^2,z^3,xy,yz,zx)$, where $k$ is an infinite field, and $M=A/(w)$. By Example~\ref{ex1}, $M$ has finite projective dimension over $A$. Moreover, $M$ is a Cohen-Macaulay $A$-module of dimension $2$. We have $G(A)\cong k[X,Y,Z,W,U]/(X^2,Y^2,Z^3,XY,YZ,ZX)$ with $c=\operatorname{reg} G(A)=2$, and
$G(M)\cong k[X,Y,Z,U]/(X^2,Y^2,Z^3,XY,YZ,ZX)$. The Hilbert series of $M$ is
$\displaystyle H_M(z)=\frac{1+3z+z^2}{(1-z)^2}$. Hence, $h_M(z)=1+3z+z^2$, and therefore
$e_1(M)=5=\binom{3}{2}+2=\binom{c+1}{2}+2$ and
$e_2(M)=1=\binom{c+1}{3}$.

\end{example}

We have not been able to construct an example satisfying equality in non Cohen-Macaulay cases.

\section{Cohen-Macaulay modules over strict complete intersection rings}
In this section, we discuss the results of Theorems \ref{main+1} and \ref{main} to the setting where $A$ is a strict complete intersection ring. Our aim is to obtain analogous bounds and structural consequences in this broader setting, without imposing the assumption that $M$ has finite projective dimension.

\begin{point}\label{complexity}
   \normalfont
   Let $A$ be a local ring and $M$ be an $A$-module. Let $\beta_i^A(M)=\lambda_A(\operatorname{Tor}_i(M,k))$ be the $i$-th Betti number of $M$ over $A.$ The complexity of $M$ over $A$ is defined as $$\textrm{cx}_A(M)=\inf \left\{b\in \N \quad \vline \quad {\limsup_{n\to \infty}} \,\, \frac{\beta_n^A(M)}{n^{b-1}}< \infty \right\}.$$
   Note that $M$ has bounded Betti numbers iff $\textrm{cx}_A(M)\leq 1.$
\end{point}

\begin{remark}
    Let $(A,\m,k)$ be a complete intersection ring with codimension $l$. Let $M$ be a finite $A$-module. It is well know that $\cx_A(M)\leq \cx_A(k)=l.$
\end{remark}

Let $(A,\m)$ be a local ring. For $x\in \m^i\setminus \m^{i+1}$ we set $\text{ord}(x)=i$ as order of $x.$ The next result is a generalization of Theorem $5.8$ of \cite{Quasipure}.

\begin{proposition}\label{ci}
    Let  $(Q,\n)$ be a regular local ring with $f_1,\ldots ,f_l\in \n^2$ of order $s$ such that $f_1^*,\ldots ,f_l^* $ is a $G(Q)$-regular sequence.  Let $A=Q/(f_1,\ldots ,f_l)$ and $M$ be Cohen Macaulay $A$-module of dimension $r$ with $\operatorname{cx}_A(M)=t< l$. Then $e_0\geq \mu(M)+c$ and $e_1(M)\geq \binom{c+1}{2}$, where $c=(l-r)(s-1).$ Further, if $e_1(M)$ attains its lower bound then $G(M)$ is Cohen-Macaulay.
\end{proposition}
\begin{proof}
    By Theorem 5.7 of \cite{Quasipure}, we get that $A=T/(g_1,\ldots ,g_t)$ and $M$ has finite projective dimension as $T$-module,  where $T=Q/(g_{t+1},\ldots ,g_l)$ and $(g_1^*,\ldots ,g_l^*)$ is a $G(Q)$-regular sequence. Since  $\text{ord}(g_i)=s$ for all $1\leq i\leq l$, we obtain that $\operatorname{reg}G(T)=(l-t)(s-1)=c$. By Lemma \ref{lowerbound}, we have $e_0\geq \mu(M)+c$ and $e_1(M)\geq \binom{c+1}{2}$. Furthermore, If $e_1(M)= \binom{c+1}{2}$ then by theorem \ref{reg}, $G(M)$ is Cohen-Macaulay $G(T)$-module (and so a Cohen-Macaulay $G(A)$-module).
\end{proof}

We now give an example in which $e_1(M)$ attains its lower bound in Proposition \ref{ci}. 

\begin{example}
\normalfont
     Let $Q=k[[x,y,z]]$, where $k$ is an infinite field, and set $A=Q/(x^2,y^2).$ Then $l=2$ and $s=2$. Let
$M=A/(x).$ Since $M\cong k[[y,z]]/(y^2)$, the module $M$ is Cohen-Macaulay. Moreover, $M$ has a periodic free resolution with Betti numbers $\beta_i^A(M)=1$ for all $i\geq 0$, and hence
$\operatorname{cx}_A(M)=1.$ Therefore
$c=(l-\operatorname{cx}_A(M))(s-1)=(2-1)(2-1)=1.$ Also,
$G(M)\cong k[Y,Z]/(Y^2),$ so $G(M)$ is Cohen-Macaulay and its Hilbert series is $\displaystyle H_M(z)=\frac{1+z}{1-z}.$ Thus
$e_0(M)=2=\mu(M)+c$ and $ e_1(M)=1=\binom{c+1}{2}.$
\end{example}

Next, We study the cases when $e_1(M)=\binom{c+1}{2}+i$, where $i=1,2$.
\begin{theorem}\label{ci1}
(with the same hypothesis as in Proposition \ref{ci}) Let $e_1(M)= \binom{c+1}{2}+1$. Then the following hold.
\begin{enumerate}
    \item If $G(M)$ is Cohen-Macaulay, then $e_2(M)=\binom{c+1}{3}$.
    \item If $G(M)$ is not Cohen-Macaulay, then $e_2(M)\leq \binom{c+1}{3}+(c-1).$ Further, if equality holds, then $\depth G(M)=r-1.$
\end{enumerate}
\end{theorem}
\begin{proof}
    By Theorem 5.7 of \cite{Quasipure}, we get that $A=T/(g_1,\ldots ,g_t)$ and $M$ has finite projective dimension as $T$-module, where $T=Q/(g_{t+1},\ldots ,g_l)$ and $(g_1^*,\ldots ,g_l^*)$ is a $G(Q)$-regular sequence. Note that for $t=l$, we may choose $f_i=g_i$ for all $1\leq i\leq l.$ Since  $\text{ord}(g_i)=s$ for all $1\leq i\leq l$, we obtain that $\operatorname{reg} G(T)=(l-t)(s-1)=c$. Since the residue field of $A$ is infinite, so is $T$. Set $\m$ and $\m_T$ be a maximal ideals of $A$ and $T$ respectively.
    
    Let $x_1,\ldots ,x_{r-2},x,y$ be $A\oplus M$-superficial sequence. Since $g_1^*,\ldots,g_t^*$ is a $G(T)$-regular sequence and $G(A)\cong G(T)/(g_1^*,\ldots,g_t^*)$, the initial forms of $x_1,\ldots,x_{r-2},x,y$ form a regular sequence on $G(T)$; hence they constitute a $T$-superficial sequence. Moreover, $(g_1,\ldots,g_t)M=0$, so the $\mathfrak m_T$-adic and $\mathfrak m$-adic filtrations on $M$ coincide, and therefore the sequence is $T\oplus M$-superficial. Set $(B,\n)=(T/(x_1,\ldots ,x_{r-2},x),\m_T/(x_1,\ldots ,x_{r-2},x))$, $L=M/(x_1,\ldots ,x_{r-2},x)M$, and $N=L/yL$. The result now follows from Theorem \ref{main+1}.
\end{proof}

\begin{theorem}\label{ci2}
(with the same hypothesis as in Proposition \ref{ci}) Let $e_1(M)= \binom{c+1}{2}+2$. Then the following hold.
\begin{enumerate}
    \item If $G(M)$ is Cohen-Macaulay, then $e_2(M)=\binom{c+1}{3}$.
    \item If $G(M)$ is not Cohen-Macaulay, then $e_2(M)\leq \binom{c+1}{3}+2(c-1).$
\end{enumerate}
\end{theorem}
\begin{proof}
   The result follows by combining the arguments in the proofs of Theorems \ref{ci1} and \ref{main}.
\end{proof}

\begin{corollary}
    (with the same hypothesis as in Proposition \ref{ci}) Let $\cx_A(M)=l$. If $e_1(M)=\binom{c+1}{2}+i$, where $i=1,2$, then $G(M)$ is Cohen-Macaulay and $e_2(M)=0$. 
\end{corollary}

\begin{proof}
    Since $\cx_A(M)=l$, we can choose $T=Q$ and  $g_i=f_i$ for all $1\leq i\leq l$ in Theorem 5.7 of \cite{Quasipure}. It follows that $c=\reg G(T)=0$. The result now follows from Corollary \ref{reg2}.
\end{proof}

We provide examples illustrating the Cohen-Macaulay case of Theorems \ref{ci1} and \ref{ci2}.

\begin{example}
\begin{enumerate}
\normalfont
    \item Let $Q=k[[x,y,z,u]]$, where $k$ is an infinite field, and set $A=Q/(x^2,y^2,z^2).$ Then $l=3$ and $s=2$. Let
$M=A/(x).$ Since $M\cong k[[y,z,u]]/(y^2,z^2)$, the module $M$ is Cohen-Macaulay. Moreover, $M$ has a periodic free resolution with Betti numbers $\beta_i^A(M)=1$ for all $i\geq 0$, and hence
$\operatorname{cx}_A(M)=1.$ Therefore
$c=(l-\operatorname{cx}_A(M))(s-1)=(3-1)(2-1)=2.$ Also,
$G(M)\cong k[Y,Z,U]/(Y^2,Z^2),$ so $G(M)$ is Cohen-Macaulay and its Hilbert series is $\displaystyle H_M(z)=\frac{(1+z)^2}{1-z}=\frac{1+2z+z^2}{1-z}.$ Thus $e_1(M)=4=\binom{c+1}{2}+1$ and $e_2(M)=1=\binom{c+1}{3}.$ Thus, it satisfies the Cohen-Macaulay case of Theorem \ref{ci1}.

\item Let $Q=k[[x,y,z]]$, where $k$ is an infinite field, and set $A=Q/(x^2,y^2).$ Then $l=2$ and $s=2$. Let
$N=A/(x)$ and $M=N^{ 3}.$ Since $N\cong k[[y,z]]/(y^2)$, the module $M$ is Cohen-Macaulay. Moreover, $\beta_i^A(N)=1$ for all $i\geq0$, and hence $\operatorname{cx}_A(M)=1.$ Therefore
$c=(l-\operatorname{cx}_A(M))(s-1)=(2-1)(2-1)=1.$ Also,
$G(M)\cong (k[Y,Z]/(Y^2))^{3}$ and
$\displaystyle H_M(z)=\frac{3+3z}{1-z}.$ Thus
$e_1(M)=3=\binom{c+1}{2}+2$ and $e_2(M)=0=\binom{2}{3}.$ Thus, it satisfies the Cohen-Macaulay case of Theorem \ref{ci2}.
\end{enumerate}

\end{example}
%\section*{Acknowledgements}
%The authors thank the referee for carefully reading the manuscript and for many pertinent comments.

%\bibliographystyle{amsplain}
%\bibliography{Samarendra}
\providecommand{\bysame}{\leavevmode\hbox to3em{\hrulefill}\thinspace}
\providecommand{\MR}{\relax\ifhmode\unskip\space\fi MR }
 %\MRhref is called by the amsart/book/proc definition of \MR.
\providecommand{\MRhref}[2]{
  \href{http://www.ams.org/mathscinet-getitem?mr=#1}{#2}
}

\end{document}